\documentclass[letterpaper,10pt,reqno,onefignum,onetabnum]{amsart}
\usepackage[english]{babel}
\usepackage{amsmath}
\usepackage{amsthm}
\usepackage{verbatim}
\usepackage{mathrsfs}
\usepackage{bm} 
\usepackage[dvipsnames]{xcolor}
\usepackage{hyperref}
\hypersetup{
     colorlinks = true,
     linkcolor = OliveGreen,
     anchorcolor = OliveGreen,
     citecolor = OliveGreen,
     filecolor = OliveGreen,
     urlcolor = OliveGreen
     }
\usepackage{algorithm}% http://ctan.org/pkg/algorithms
\usepackage{algorithmic}% http://ctan.org/pkg/algorithms
\usepackage{float}
\usepackage{lipsum}
\usepackage{amsfonts}
\usepackage{amssymb}
\usepackage{graphicx}
\usepackage{epstopdf}
\usepackage{cases}
\usepackage{multirow}
\ifpdf
  \DeclareGraphicsExtensions{.eps,.pdf,.png,.jpg}
\else
  \DeclareGraphicsExtensions{.eps}
\fi
\usepackage{verbatim}
\usepackage{mathrsfs}
\usepackage{bm}

\newtheorem{theorem}{Theorem}[section]
\newtheorem{proposition}[theorem]{Proposition}

\newtheorem{corollary}[theorem]{Corollary}

\newtheorem{remark}[theorem]{Remark}
\newtheorem{example}[theorem]{Example}

\usepackage{lipsum}
\usepackage{amsfonts}
\usepackage{amssymb}
\usepackage{graphicx}
\usepackage{epstopdf}
\usepackage{cases}
\usepackage{multirow}
\usepackage{algorithmic}
\usepackage{tikz}
\usetikzlibrary{matrix}
\usetikzlibrary{arrows}
\ifpdf
  \DeclareGraphicsExtensions{.eps,.pdf,.png,.jpg}
\else
  \DeclareGraphicsExtensions{.eps}
\fi
\usepackage{verbatim}
\usepackage{mathrsfs}
\usepackage{bm} 
\usepackage{color}
\usepackage{caption}
\usepackage{subcaption}
\usepackage{enumitem}
\newcommand{\Id}{\textnormal{Id}}

\newcommand{\zer}{\textnormal{zer}}

\newcommand{\gra}{\textnormal{gra}\,}

\newcommand{\sri}{\ensuremath{\text{\rm sri}\,}}

\newcommand{\spc}{\ensuremath{\overline{\operatorname{span}}\,}}
\newcommand{\cone}{\ensuremath{\text{\rm cone\,}}}
\newcommand{\emp}{\ensuremath{{\varnothing}}}
\newcommand{\inte}{\ensuremath{\text{\rm int}}\,}

\newcommand{\scal}[2]{{\left\langle{{#1}\mid{#2}}\right\rangle}}

\newcommand{\menge}[2]{\big\{{#1}~\big |~{#2}\big\}} 
\newcommand{\pinf}{\ensuremath{{+\infty}}}

\newcommand{\RR}{\ensuremath{\mathbb{R}}}
\newcommand{\RP}{\ensuremath{\left[0,+\infty\right[}}

\newcommand{\HH}{\mathcal{H}}
\newcommand{\GG}{\mathcal G}

\newcommand{\RX}{\ensuremath{\left]-\infty,+\infty\right]}}
\newcommand{\RXX}{\ensuremath{\left[-\infty,+\infty\right]}}
\newcommand{\prox}{\ensuremath{\text{\rm prox}}}
\newcommand{\infconv}{\ensuremath{\mbox{\small$\,\square\,$}}}
\newcommand{\ran}{\ensuremath{\text{\rm ran}\,}}
\newcommand{\dom}{\ensuremath{\text{\rm dom}\,}}

\usepackage{geometry}
\numberwithin{equation}{section}

\DeclareFontEncoding{FMS}{}{}
\DeclareFontSubstitution{FMS}{futm}{m}{n}
\DeclareFontEncoding{FMX}{}{}
\DeclareFontSubstitution{FMX}{futm}{m}{n}
\DeclareSymbolFont{fouriersymbols}{FMS}{futm}{m}{n}
\DeclareSymbolFont{fourierlargesymbols}{FMX}{futm}{m}{n}
\DeclareMathDelimiter{\nr}{\mathord}{fouriersymbols}{152}{fourierlargesymbols}{147}

\DeclareMathDelimiter{\nr}{\mathord}{fouriersymbols}{152}{fourierlargesymbols}{147}

\title[Resolvent of the parallel composition]{Resolvent of the parallel 
	composition and 
	proximity operator of the infimal postcomposition}

\author{Luis M. Brice\~no-Arias \& Fernando Rold\'an}
\address{Departamento de Matem\'{a}tica, Universidad T\'{e}cnica Federico Santa Mar\'{i}a, Avenida Espa\~{n}a 1680, Valpara\'{i}so, Chile}
\email{luis.briceno@usm.cl, fernando.roldan@usm.cl}


\begin{document}

\begin{abstract} In this paper we provide the resolvent 
computation of the 
parallel composition of a maximally monotone operator by a linear 
operator under mild assumptions. Connections with a 
modification of the warped 
resolvent are provided. In the context of convex optimization,
we obtain the proximity operator of the infimal postcomposition 
of a convex function by a linear operator and we extend full range 
conditions on the linear operator to mild qualification 
conditions. We also introduce a generalization of the proximity 
operator involving a general linear bounded operator leading to a 
generalization of Moreau's decomposition for composite convex 
optimization.
\par
\bigskip

\end{abstract}

\maketitle

\section{Introduction}\label{sec1}
In this paper we aim at computing the resolvent of the 
\textit{parallel 
composition} of $A$ by $L$ \cite{Beck14}, defined by
\begin{equation}
L\rhd A=(LA^{-1}L^*)^{-1},
\end{equation}
where $\HH$ and $\GG$ are real Hilbert spaces, $A\colon\HH\to 
2^{\HH}$ and
$L\colon \HH\to\GG$ is linear and bounded. In the case when 
$\HH=H\oplus H$ for some real Hilbert space $H$, $\GG=H$, 
$A\colon (x,y)\mapsto Bx\times Cx$ for some set-valued 
operators 
$B$ and $C$ defined in $H$ and
$L\colon (x,y)\mapsto x+y$, we have $L\rhd A=B\infconv C$ 
\cite[Example~25.40]{1}, 
where $B\infconv C=(B^{-1}+C^{-1})^{-1}$ is the \textit{parallel 
sum} of $B$ and $C$, motivating the name of the operation.
The parallel composition appears naturally in composite 
monotone 
inclusions. Indeed, if $B\colon \GG\mapsto 2^{\GG}$, the dual 
inclusion associated to
\begin{equation}
\text{find}\quad x\in\HH\quad \text{such that }\quad 
0\in Ax +L^*BLx,
\end{equation}
is
\begin{equation}
	\text{find}\quad u\in\GG\quad \text{such that }\quad 
	0\in B^{-1}u+(-L\rhd A)^{-1}u.
\end{equation}
When $L^*L=\alpha\Id$ for some $\alpha\ge 0$ or when $L^*$ 
has full range, explicit 
formulas for the resolvent of $LA^{-1}L^*$ depending on the 
resolvent 
of $A$ 
can be found in \cite[Proposition~23.25]{1}. In 
\cite{Fuku96,Tang19} 
some variants and fixed point methods to compute the resolvent 
are 
proposed under full range condition on $L^*$ and a similar fixed 
point approach is used in \cite{Moudafi14} under the maximal 
monotonicity of $LA^{-1}L^*$. This 
computation is useful in \cite{Vanden20} 
 for the equivalence between the 
primal-dual \cite{cp,vu} and Douglas-Rachford 
splitting (DRS)
\cite{DR1956,Lions1967}  algorithms.

In the particular case when $A$ is the subdifferential of a convex 
function $f\colon\HH\to\RX$ satisfying dual qualification 
conditions, we have that $L\rhd A$ is the subdifferential of the 
\textit{infimal postcomposition} of $f$ by $L$, defined by
\begin{equation}
L\rhd f\colon \GG\to{\RXX}\colon u\mapsto 
\inf_{\substack{x\in\HH\\Lx=u}}f(x).
\end{equation}
This operation appears naturally when dealing with the dual of 
composite optimization problems since we 
have $(L\rhd f)^*=f^*\circ L^*$ under mild assumptions 
\cite[Proposition~13.24(iv)]{1}. 
Moreover, it is related {to} the 
parallel composition via the identities 
\begin{equation}
L\rhd (\partial f)=(L(\partial 
f^*)L^*)^{-1}=(\partial (f^*\circ L^*))^{-1}=\partial (f^*\circ 
L^*)^*=\partial (L\rhd f),
\end{equation}
where the second equality holds 
if, e.g., $0\in \sri(\dom f^*-\ran L^*)$ \cite[Corollary~16.53]{1}.
Therefore, {under previous assumption the resolvent of $L\rhd 
(\partial f)$ and the proximity 
operator of $L\rhd f$ coincide. Moreover, since the the 
\textit{alternating direction method of multipliers} (ADMM) for 
solving 
$\inf (f+g\circ L)$ is 
an application of DRS to the Fenchel-Rockafellar dual 
$\inf(f^*\circ (-L^*)+g^*)$ \cite{gabay83} (see also 
\cite{BrediesSun17,Cote17,condat2020proximal,patrinos2020,Siopt3}),
the computation of the proximity 
operator of $L\rhd f$ is relevant in the derivation of ADMM.
In the literature, several additional hypotheses have been 
assumed in 
order to ensure that the iterates of ADMM are well defined and 
that it 
converges. In particular, in 
\cite[Theorem~5.7]{condat2020proximal},
the operator $(\partial f+L^*L)^{-1}L^*$ is assumed to be 
single-valued, in \cite[Proposition~5.2]{patrinos2020} it is 
assumed to 
have full domain, and in \cite{BrediesSun17,Cote17} the 
strong monotonicity 
of $(\partial f+L^*L)$ is assumed  in 
order obtain full domain and single-valuedness.} It is worth to 
notice 
that some fixed 
point approaches and algorithms for computing $\prox_{f^*\circ 
L^*}$ are proposed in 
\cite{fadili09,Micc11} in the context of sparse recovery in image 
processing.

In this paper we derive a formula for the resolvent of 
the parallel composition and for the proximity operator of the 
infimal 
postcomposition in a real Hilbert space with non-standard metric 
under mild assumptions. This is obtained 
from a formula of the resolvent of $LA^{-1}L^*$ via the
 non-standard metric version of Moreau's identity in 
 \cite[Proposition~23.34(iii)]{1}. Our computation is related {to} 
 a modification of the warped resolvent defined in 
 \cite{Comb_warped} (see 
 \cite{Giss_warped} for a particular case). {We extend and 
 generalize \cite{Tang19} for parallel compositions and 
 \cite{BrediesSun17,Cote17,condat2020proximal,patrinos2020}
for infimal postcompositions related to ADMM and the 
well-posedness 
of its iterates.}
We also derive a generalization of Moreau's decomposition 
\cite{Moreau_prox} for 
composite maximally monotone operators and for composite 
convex optimization under standard assumptions by using a 
generalization of the proximity operator.

\section{Notation and preliminaries}
Throughout this paper $\HH$ and $\GG$ are real Hilbert spaces 
with 
the scalar 
product $\scal{\cdot}{\cdot}$ and associated norm $\|\cdot\|$. 
The 
identity operator on $\HH$ is denoted by $\Id$. 
Let $A:\HH \rightarrow 2^{\HH}$ be a set-valued operator. 
The domain of $A$ is 
$\dom\, A = \menge{x \in \HH}{Ax \neq  \varnothing}$, the range 
of $A$
is $\ran\, A = \menge{u \in 
\HH}{(\exists x \in \HH)\,\, u \in Ax}$, the graph of 
$A$ is  $\gra A = \menge{(x,u) \in \HH \times \HH}{u \in Ax}$,
the set of zeros of $A$ is  $\zer A = 
\menge{x \in \HH}{0 \in Ax}$, the inverse of $A$ is 
$A^{-1} \colon u \mapsto \menge{x \in \HH}{u \in Ax}$, and its
resolvent is $J_A=(\Id+A)^{-1}$. 
For every $D\subset\HH$, $A\!\mid_D$ is the restriction of $A$ to 
$D$, 
which
satisfies $\dom A\!\mid_D=\dom A\cap D$ and, for every $x\in 
D$, 
$A\!\mid_Dx=Ax$. The operator $A$ is \textit{injective on $D$} if
\begin{equation}
\label{e:inject}
(\forall x\in\HH)(\forall y\in\HH)\quad Ax\cap Ay\cap D\ne\emp
\quad \Rightarrow\quad x=y,
\end{equation}
and $A$ is injective if it is injective on $\HH$. It is clear that 
injectivity of $A$ on $D$ implies its injectivity on $D'$ when 
$D'\subset D$. Moreover, the operator $A$ is monotone if 
\begin{equation}\label{def:monotone}
(\forall (x,u) \in \gra A) (\forall (y,v) \in \gra A)\quad \scal{x-y}{u-v}
\geq 0,
\end{equation}
$A$ is strongly monotone if there exists $\alpha>0$ such that 
\begin{equation}\label{def:stmonotone}
(\forall (x,u) \in \gra A) (\forall (y,v) \in \gra A)\quad \scal{x-y}{u-v}
\geq \alpha\|x-y\|^2,
\end{equation}
and $A$ is maximally monotone if it is monotone and, for 
every $(x,u) \in \HH \times \HH$,
\begin{equation} \label{def:maxmonotone}
(x,u) \in \gra A \quad \Leftrightarrow\quad  (\forall (y,v) \in \gra A)\ 
\ \langle x-y \mid u-v \rangle \geq 0.
\end{equation}
For 
every strongly monotone self-adjoint linear bounded 
operator $U\colon\HH\to\GG$, we denote
$\scal{\cdot}{\cdot}_U=\scal{\cdot}{U\cdot}$ and 
$\|\cdot\|_U=\sqrt{\scal{\cdot}{\cdot}_U}$, which define an inner 
product and the associated norm in $\HH$, respectively.

We denote by $\Gamma_0(\HH)$ the class of proper lower 
semicontinuous convex functions $f\colon\HH\to\RX$. Let 
$f\in\Gamma_0(\HH)$.
The Fenchel conjugate of $f$ is 
defined by $f^*\colon u\mapsto \sup_{x\in\HH}(\scal{x}{u}-f(x))$,
$f^*\in \Gamma_0(\HH)$,
the subdifferential of $f$ is the maximally monotone operator
\begin{equation}
\partial f\colon x\mapsto \menge{u\in\HH}{(\forall y\in\HH)\:\: 
f(x)+\scal{y-x}{u}\le f(y)},
\end{equation}
$(\partial f)^{-1}=\partial f^*$, the set of 
minimizers of $f$ is denoted by $\arg\min_{x\in \HH}f(x)$, 
and we have  that $\zer\,(\partial f)=\arg\min_{x\in \HH}f(x)$. 
Given a strongly monotone self-adjoint linear operator 
$U\colon\HH\to\HH$, we denote by 
\begin{equation}
	\label{e:prox}
\prox^{U}_{f}\colon 
x\mapsto\arg\min_{y\in\HH}\big(f(y)+\frac{1}{2}\|x-y\|_{U}^2\big),
\end{equation}
and by $\prox_{f}=\prox^{\Id}_{f}$. We have 
\cite[Proposition~24.24]{1} (see also \cite[Section~3]{Combvu14})
\begin{equation}
\label{e:proxU}
\prox^{U}_f=U^{-\frac{1}{2}}\prox_{f\circ 
U^{-\frac{1}{2}}}U^{\frac{1}{2}}=J_{U^{-1}\partial f}
\end{equation}
and it is single-valued since the objective function in 
\eqref{e:prox} is 
strongly convex.
Moreover, it follows from 
\cite[Proposition~23.34(iii)]{1} that
\begin{equation}
\label{e:Moreauresns}
J_{UA}+UJ_{U^{-1}A^{-1}}U^{-1}=\Id,
\end{equation}
and, in the case of convex functions,
\cite[Proposition~24.24]{1} yields
\begin{equation}
\label{e:Moreau_nonsme}
\prox^{U}_{f}=
\Id-U^{-1}\,  
\prox^{U^{-1}}_{f^*}\, 
U=U^{-1}\,(\Id-  
\prox^{U^{-1}}_{f^*})\, 
U.
\end{equation}
Given a non-empty set $C\subset\HH$, we denote by $\spc C$
the closed span of $C$, by $\cone\, C$ its conical hull. 
Let $C$ be a non-empty closed convex subset of $\HH$. We 
denote 
by 
$\sri C=\menge{x\in C}{\cone(C-x)=\spc(C-x)}$ its strong 
relative interior, by
$\iota_C\in\Gamma_0(\HH)$ the indicator function of $C$, which 
takes the value $0$ in $C$ and $\pinf$ otherwise, 
by $P^U_C=\prox^U_{\iota_C}$ the projection onto $C$ with 
respect 
to $(\HH,\scal{\cdot}{\cdot}_U)$, and we denote 
$P_C=P_C^{\Id}$. 
It follows from \eqref{e:proxU} that
\begin{equation}
\label{e:PU}
P^U_C=U^{-\frac{1}{2}}\prox_{\iota_C\circ 
U^{-\frac{1}{2}}}U^{\frac{1}{2}}=
U^{-\frac{1}{2}}P_{U^{\frac{1}{2}}C}U^{\frac{1}{2}}.
\end{equation}
Given a 
linear bounded operator $L:\HH \to \GG$, we denote its adjoint 
by $L^*\colon\GG\to\HH$, its kernel (or null space) by $\ker L$, 
its 
range by 
$\ran L$, and, if $\ran L$ is closed, its \textit{Moore-Penrose 
inverse}
by 
\begin{equation}
L^{\dagger}\colon \GG\to \HH\colon y\mapsto P_{C_y}0,
\end{equation}
where $C_y=\{x\in\HH\,\mid\,
L^*Lx=L^*y\}$. If $L^*L$ is invertible, we have 
\cite[Example~3.29]{1}
\begin{equation}
\label{e:moinv}
L^{\dagger}=(L^*L)^{-1}L^*.
\end{equation}
For definitions and properties of monotone operators,
nonexpansive mappings, and convex analysis, the 
reader is referred to \cite{1}. 

We now introduce a
modification of the warped resolvent introduced in 
\cite{Comb_warped}  
(see also \cite{Giss_warped} for a particular case and 
applications). 
Let $A\colon\HH\to 2^{\HH}$ be a 
set-valued operator and let $K\colon\HH\to\HH$.
The \textit{warped resolvent of $A$ with kernel $K$} is 
defined by $J_A^K=(K+A)^{-1}K$. 
In the case when $K$ is {linear and} invertible, we have 
\begin{equation}
	\label{e:warpinvert}
	J_A^K=(K+A)^{-1}K=(K(\Id+K^{-1}A))^{-1}K=J_{K^{-1}A},
\end{equation}
which has full domain and it is single-valued if $K^{-1}A$
is maximally monotone. 
The following result characterizes 
the full domain and single-valuedness of $J_A^K$ in a general 
context.
\begin{proposition}
\label{p:domsv}
Let $A\colon\HH\to 2^{\HH}$ be a 
set-valued operator and let $K\colon\HH\to\HH$. Then the 
following 
holds.
\begin{enumerate}[label=(\roman*)]
\item
\label{p:domsvi}
$\dom J_A^K=\HH$ $\:\:\Leftrightarrow\:\:$ $\ran K\subset 
\ran(K+A)$.
\item 
\label{p:domsvii}
$J_A^K$ is {at most} single-valued $\:\:\Leftrightarrow\:\:$ $K+A$ 
is 
injective 
on 
$\ran K$.
\end{enumerate}
\end{proposition}
\begin{proof}
\ref{p:domsvi}: For every $x\in\HH$ we have
\begin{align}
x\in\dom J_A^K\quad &\Leftrightarrow\quad (\exists 
u\in\HH)\quad  
u\in 
(K+A)^{-1}Kx\nonumber\\
&\Leftrightarrow\quad (\exists u\in\HH)\quad  
Kx\in(K+A)u\nonumber\\
&\Leftrightarrow\quad Kx\in\ran(K+A),
\end{align}
and the result follows.
\ref{p:domsvii}: First assume that $J_A^K$ is {at most} 
single-valued. In view 
of \eqref{e:inject}, let $x$ and 
$y$ in $\HH$ and suppose that
there exists $z\in\HH$ such that $Kz\in (Kx+Ax)\cap(Ky+Ay)$.
Then $\{x\}\cup\{y\}\subset J_A^Kz$ and single-valuedness of 
$J_A^K$
implies $x=y$, which yields the injectivity on $\ran K$. 
Conversely, let 
$z\in \dom J_A^K$ and let $x$ and $y$ in $J_A^Kz$. Then,
$Kz\in (Kx+Ax)\cap (Ky+Ay)\cap \ran K$ and injectivity on $\ran K$
implies $x=y$.
\end{proof}
In \cite[Definition~1.1]{Comb_warped} it is assumed that 
$K+A$ is injective in the whole space in order to guarantee 
that $J_A^K$ is single-valued, but this is a stronger assumption in 
general, as the following example illustrates.
\begin{example}
Let $\alpha>0$, set $\HH=\RR$, set $K\colon x\mapsto {\rm 
med} 
\{-1,x,1\}$ be the median of real values $x$, $-1$, and $1$, and 
set $A=\alpha K$. Note that 
$A$ and $K$ are maximally monotone, single-valued, and 
$\ran K=[-1,1]\subset [-1-\alpha,1+\alpha]=\ran(K+A)$. Moreover,
observe that $K+A=(1+\alpha){\rm med} 
{\{-1,\cdot,1\}}$ is injective on $\ran K$
but it is not injective on $\RR$, since $(K+A)1=(K+A)2=1+\alpha$.
\end{example}
The \textit{warped proximity operator of  $f$ with kernel $K$}
is defined by 
\begin{equation}
\label{e:proxwarp}
\prox^K_{f}=J^K_{\partial f}=(K+\partial f)^{-1}K
\end{equation}
and note that
it coincides with \eqref{e:proxU} when $K$ is strongly monotone,
self-adjoint, linear, and bounded, in view of \eqref{e:warpinvert}.
\section{Resolvent of parallel composition}
\label{sec2}
The following result is a generalization of 
\cite[Proposition~23.25]{1} and provides an explicit computation 
of the 
resolvent of $U M^* B M$ under mild assumptions.
\begin{theorem}
\label{p:rescalc}
Let $\HH$ and $\GG$ be real Hilbert spaces, let
$B\colon\HH\to 2^{\HH}$ be a maximally monotone operator, let 
$M\colon\GG\to\HH$ be a linear bounded operator such that 
 $M^*BM$ is maximally monotone in $\GG$, and let 
 $U\colon\GG\to\GG$ be a 
 $\mu-$strongly monotone self-adjoint linear operator for some 
 $\mu>0$. 
Then $U M^*BM$ is maximally 
monotone in $(\GG,\scal{\cdot}{\cdot}_{U^{-1}})$ and the 
following assertions hold:
\begin{enumerate}[label=(\roman*)]
\item 
\label{p:rescalc1}
$\ran M\subset \dom(M U M^*+B^{-1})^{-1}$ and 
\begin{equation}
J_{U M^*BM}= 
\Id-U 
M^*(M U M^*+B^{-1})^{-1} M.
\end{equation}

\item 
\label{p:warpedi}
$\ran (MUM^*)\subset \ran (MUM^*+B^{-1})$.

\item 
\label{p:warpedii}
$(MUM^*+B^{-1})\mid_{\ran M}$ is injective.

\item 
\label{p:warpediii}
Suppose that 
$\ran M$ is 
closed. Then
\begin{equation}
J_{UM^*BM}=\Id-U 
M^*J_{B^{-1}}^{MUM^*}(\sqrt{U}M^*)^{\dagger}
\sqrt{U}^{-1}.
\end{equation}

\item 
\label{p:rescalc2}
Suppose that $\ran M=\HH$. Then
\begin{align}
J_{U M^*BM}&	=\Id-U 
M^*J_{(MU M^*)^{-1} B^{-1}}(MU 
M^*)^{-1}M\label{e:fineq}\\
&=P^{U^{-1}}_{\ker M}+U M^*(MU 
M^*)^{-1}J_{MU M^*B}M.\label{e:sineq}
\end{align}
\end{enumerate}
\end{theorem}
\begin{proof} The maximal monotonicity of $U M^*BM$ 
follows from 
\cite[Proposition~20.24]{1}. 
\ref{p:rescalc1}:
For every $x$ and $p$ in $\GG$, we 
have
\begin{align}
p=J_{U M^*BM}x\quad 
&\Leftrightarrow\quad x-p\in U M^*BMp\\
&\Leftrightarrow\quad (\exists v\in \HH)\:\: 
\begin{cases}
x-p=U M^*v\\
v\in BMp
\end{cases}\nonumber\\
&\Leftrightarrow\quad (\exists v\in \HH)\:\: 
\begin{cases}
p=x-U M^*v\\
Mp\in B^{-1}v
\end{cases}\nonumber\\
&\Leftrightarrow\quad (\exists v\in \HH)\:\: 
\begin{cases}
p=x-U M^*v\\
Mx\in MU M^*v+B^{-1}v.
\end{cases}\nonumber\\
&\Leftrightarrow\quad (\exists v\in(M U 
M^*+B^{-1})^{-1}Mx)\quad  
p=x-U M^*v,
\end{align}
and the result follows.
\ref{p:warpedi}: It follows from 
\ref{p:rescalc1} that
\begin{equation}
\ran (MUM^*)\subset \ran M\subset 
\dom(MUM^*+B^{-1})^{-1}=\ran (MUM^*+B^{-1}).
\end{equation}
\ref{p:warpedii}: Let $x$ and $y$ in $\ran M$ be such that
there exists $u\in (MUM^*x+B^{-1}x)\cap(MUM^*y+B^{-1}y)$.
Then, $u-MUM^*x\in B^{-1}x$, $u-MUM^*y\in B^{-1}y$, and the 
monotonicity of $B^{-1}$ yields
\begin{align}
\label{e:auxwarped}
0&\le \scal{-MUM^*(x-y)}{x-y}\nonumber\\
&=-\scal{UM^*(x-y)}{M^*(x-y)}\nonumber\\
&\le -\mu\|M^*(x-y)\|^2,
\end{align}
which implies $x-y\in\ker M^*$. Since $x-y\in
\ran M\subset\overline{\ran M}$, it follows from 
\cite[Fact~2.25(iv)]{1} 
that $x-y\in \ker 
M^*\cap \overline{\ran M}=\{0\}$, which yields the result.

\ref{p:warpediii}: 
Denote by $\GG_U$ the Hilbert space $\GG$ endowed with 
the 
scalar product $\scal{\cdot}{\cdot}_{U^{-1}}$.
Note that $M^{*_U}=UM^*U^{-1}$, where $M^{*_U}$ and 
$M^*$ are the adjoints of $M$ in $\GG_U$ and $\GG$, 
respectively.
Moreover, \cite[Fact~2.25(iv)]{1} and the closedness of $\ran M$ 
on 
$\GG_U$ 
yield $(\ker M)^{\bot_U}=\ran M^{*_U}
=\ran (UM^*U^{-1})=\ran (UM^*)$, where $\bot_U$ stands for the 
orthogonal complement in $\GG_U$. Hence, 
\begin{equation}
\label{e:orthdecU}
	\GG_U=\ker M\oplus \ran(UM^*)
\end{equation}
is an orthogonal decomposition of $\GG_U$. Hence, we have 
from 
\cite[Proposition~24.24(ii) \& Proposition~3.30(iii)]{1} that
\begin{align}
\label{e:pkerU}
	P^{U^{-1}}_{\ker M}&=\prox_{\iota_{\ker 
	M}}^{U^{-1}}\nonumber\\
	&=\sqrt{U}\prox_{\iota_{\ker M}\circ 
		\sqrt{U}}\sqrt{U}^{-1}\nonumber\\
	&=\sqrt{U}P_{\ker (M\sqrt{U})}\sqrt{U}^{-1}\nonumber\\
	&=
	\Id-UM^*(\sqrt{U}M^*)^{\dagger}\sqrt{U}^{-1},
\end{align}
and
\begin{equation}
\label{e:pranU}
	P^{U^{-1}}_{\ran 
	(UM^*)}=UM^*(\sqrt{U}M^*)^{\dagger}\sqrt{U}^{-1},
\end{equation}
where $(\sqrt{U}M^*)^{\dagger}$ 
is the Moore-Penrose inverse of  $\sqrt{U}M^*\colon\HH\to\GG$.
Therefore, \ref{p:rescalc1} 
asserts that 
\begin{align}
	\label{e:moreaures}
	J_{UM^*BM}&=\Id-U 
	M^*(B^{-1}+M U M^*)^{-1} M\nonumber\\
	&=\Id-U 
	M^*(B^{-1}+M U M^*)^{-1} MP^{U^{-1}}_{\ran 
	(UM^*)}\nonumber\\
	&=\Id-U 
	M^*(B^{-1}+M U M^*)^{-1} 
	MUM^*(\sqrt{U}M^*)^{\dagger}\sqrt{U}^{-1}\nonumber\\
	&=\Id-U 
	M^*J_{B^{-1}}^{MUM^*}(\sqrt{U}M^*)^{\dagger}
	\sqrt{U}^{-1},
\end{align}
where in the last equality $J_{B^{-1}}^{MUM^*}$ has full domain 
in 
view of \ref{p:warpedi} and 
Proposition~\ref{p:domsv}\ref{p:domsvi}. 

\ref{p:rescalc2}: Since $\ran M=\HH$ is closed and $U$ is 
$\mu-$strongly 
monotone for some $\mu>0$,
 $MU M^*$ is strongly monotone and, thus, invertible. Indeed, 
 for every $v\in 
 \HH$, 
 \cite[Fact~2.26]{1} implies that there exists $\alpha>0$ such that
\begin{equation}
\scal{MU M^*v}{v}=\scal{U M^*v}{M^*v}\ge 
\mu\|M^*v\|^2\ge \mu\alpha^2\|v\|^2.
\end{equation}
Hence, since $(\sqrt{U}M^*)^*(\sqrt{U}M^*)=MUM^*$,
\eqref{e:fineq} follows from \ref{p:warpediii},
\eqref{e:warpinvert}, and \eqref{e:moinv}. Moreover, since 
\eqref{e:pkerU} and \eqref{e:moinv} yield
$P^{U^{-1}}_{\ker M}=\Id-UM^*(MUM^*)^{-1}M$,
 \eqref{e:sineq}  follows from \eqref{e:fineq} and
\eqref{e:Moreauresns}.
\end{proof}
\begin{remark}
\begin{enumerate}[label=(\roman*)]
\item Note that Theorem~\ref{p:rescalc}\ref{p:rescalc1}
provides the existence of zeros of the monotone operator
$MU M^*+B^{-1}$ from the maximal monotonicity of 
$M^* BM$, which is guaranteed, e.g., if  $\cone(\ran 
M-\dom B)=\spc(\ran 
M-\dom B)$ \cite[Corollary~25.6]{1} (see \cite{Bot_grad} for a 
weaker {assumption} involving the domain of the Fitzpatrick 
function).
\item 
Note that, from Theorem~\ref{p:rescalc}\ref{p:rescalc1}, 
$M^*(M U M^*+B^{-1})^{-1} M\colon \GG\to\GG$ is 
single-valued, 
even if 
$(M U M^*+B^{-1})^{-1}$
can be a set-valued mapping. Indeed,
for every $x\in\GG$,  let $v$ and $w$ in $(M U 
M^*+B^{-1})^{-1}Mx$. Then, 
$M(x-U M^*v)\in B^{-1}v$ and $M(x-U M^*w)\in 
B^{-1}w$ and the monotonicity of $B^{-1}$ yields
\begin{equation}
0\le 
\scal{-MU(M^*v-M^*w)}{v-w}=-\|M^*v-M^*w\|_{U}^2,
\end{equation}
which implies $M^*v=M^*w$. This computation is consistent with 
the 
fact that the resolvent of the monotone operator $U M^* 
BM$ is single-valued.
\item
Observe that Theorem~\ref{p:rescalc}\ref{p:warpedi}
and Proposition~\ref{p:domsv}\ref{p:domsvi} imply that
$\dom J_{B^{-1}}^{MUM^*}=\HH$. On the other hand, the 
single-valuedness 
of $J_{B^{-1}}^{MUM^*}$ is not guaranteed since 
$MUM^*+B^{-1}$
is not necessarily injective on $\ran (MUM^*)$ (see 
Proposition~\ref{p:domsv}\ref{p:domsvii}). Indeed, suppose 
that $\ker M^*\neq\{0\}$ and that $B^{-1}=N_C$, where $C$ is 
the 
closed 
ball centered at $0$
with radius $1$. By taking $x=0$ and $y\in(\ker 
M^*\smallsetminus\{0\})\cap \inte C$, we have $\{0\}=N_Cx\cap 
N_Cy=(MUM^*x+B^{-1}x)\cap (MUM^*y+B^{-1}y)$, $0\in\ran 
MUM^*$,
and $x\ne y$. {Since $0=MUM^*0$, this implies $\{x,y\}\subset 
J^{M^*UM}_{B^{-1}}0$ and, thus, $J^{M^*UM}_{B^{-1}}$ is not 
single-valued.}
However, when $\ran M$ is closed, it follows from 
Theorem~\ref{p:rescalc}\ref{p:warpedii} and $\ran 
(\sqrt{U}M^*)^{\dagger}=\ran (M\sqrt{U})=\ran M$ 
\cite[Proposition~3.30(v)]{1} that
$J_{B^{-1}}^{MUM^*}(\sqrt{U}M^*)^{\dagger}$ is single-valued.

\item 
In the particular case when $U=\Id$, 
Theorem~\ref{p:rescalc}\ref{p:rescalc2}
coincides with \cite[Proposition~23.25]{1}. On the other hand, 
when
 $M=\Id$ and $\HH=\GG$, 
we recover from Theorem~\ref{p:rescalc}\ref{p:rescalc2} the 
Moreau's 
decomposition with non-standard metric 
in \cite[Proposition~23.34(iii)]{1} recalled in \eqref{e:Moreauresns}.
\end{enumerate}
\end{remark}
We conclude this section with the computation of the resolvent of 
the parallel composition $L\rhd A$.
\begin{corollary}
\label{c:parallel}
Let $A\colon\HH\to 2^{\HH}$ be a maximally monotone operator, 
let $L\colon\HH\to\GG$ be a linear bounded operator such  that 
$LA^{-1}L^*$ is maximally monotone in $\GG$. Moreover, let 
$U\colon \GG\to\GG$ be a self-adjoint strongly monotone linear 
bounded operator.
Then, $U(L\rhd A)$ is maximally monotone in 
$(\GG,\scal{\cdot}{\cdot}_{U^{-1}})$ and the following holds:
\begin{enumerate}[label=(\roman*)]
\item 
\label{c:paralleli}
$J_{U(L\rhd A)}=L(A+L^*U^{-1}L)^{-1}L^*U^{-1}.$
\item 
\label{c:parallelii}
Suppose that $\ran L$ is closed. Then,
\begin{equation}
J_{U(L\rhd 
A)}=LJ^{L^*U^{-1}L}_{A}(\sqrt{U}^{-1}L)^{\dagger}\sqrt{U}^{-1}.
\end{equation}
\item 
\label{c:paralleliii}
Suppose that $\ran L^*=\HH$. Then,
\begin{equation}
J_{U(L\rhd 
A)}=LJ_{(L^*U^{-1}L)^{-1}A}(L^*U^{-1}L)^{-1}L^*{U}^{-1}.
\end{equation}
\end{enumerate}
\end{corollary}
\begin{proof}
Since $L\rhd A=(LA^{-1}L^*)^{-1}$, the maximal monotonicity of 
$U(L\rhd A)$
follows from 
\cite[Propositions~20.22 \& 20.24]{1}.
\ref{c:paralleli}:
By applying Theorem~\ref{p:rescalc}\ref{p:rescalc1} to 
$B=A^{-1}$ 
and $M=L^*$, it follows from \eqref{e:Moreauresns} that
\begin{align}
J_{U(L\rhd A)}&=\Id-UJ_{U^{-1}LA^{-1}L^*}U^{-1}\nonumber\\
&=\Id-U(\Id-U^{-1}L(A+L^*U^{-1}L)^{-1}L^*)U^{-1}\nonumber\\
&=L(A+L^*U^{-1}L)^{-1}L^*U^{-1}.
\end{align}
\ref{c:parallelii}: By applying 
Theorem~\ref{p:rescalc}\ref{p:warpediii} to $B=A^{-1}$ 
and $M=L^*$, we obtain
\begin{align}
J_{U(L\rhd A)}&=\Id-UJ_{U^{-1}LA^{-1}L^*}U^{-1}\nonumber\\
&=\Id-U\big(\Id-
U^{-1}LJ^{L^*U^{-1}L}_{A}(\sqrt{U}^{-1}L)^{\dagger}\sqrt{U}\big)U^{-1}\nonumber\\
&=LJ^{L^*U^{-1}L}_{A}(\sqrt{U}^{-1}L)^{\dagger}\sqrt{U}^{-1}.
\end{align}
\ref{c:paralleliii}: As in the proof of
Theorem~\ref{p:rescalc}\ref{p:warpediii}, $L^*U^{-1}L$ is strongly 
monotone and, hence, invertible, and the result follows from 
\ref{c:parallelii}, \eqref{e:warpinvert}, and  \eqref{e:moinv}.
\end{proof}

\section{Proximity operator of the infimal 
postcomposition}\label{sec3}
For every 
$f\in\Gamma_0(\HH)$, every linear bounded operator $L\colon 
\HH\to 
\GG$, 
and every strongly monotone self-adjoint linear bounded operator
$U\colon\GG\to\GG$, define 
\begin{equation}
\label{e:proxgen}
\prox_{f,L}^{U}\colon\GG\to 2^{\HH}\colon
u\mapsto\arg\min_{x\in\HH}
\left(f(x)+\frac{1}{2}\|Lx-u\|^2_{U}\right).
\end{equation}
Note that \cite[Theorem~16.3 \& Theorem~16.47(i)]{1} 
yield 
\begin{align}
\label{e:moreoureck}
(\forall u\in\GG)(\forall x\in\HH)\quad  
x\in\prox_{f,L}^Uu\quad&\Leftrightarrow\quad  
	0\in\partial f(x)+L^*U(Lx-u)\nonumber\\
&\Leftrightarrow\quad  
	x\in(\partial f+L^*UL)^{-1}L^*Uu.
\end{align}
When $L=\Id$, we have $\prox_{f,\Id}^{U}=\prox_{f}^{U}$ and it is 
single-valued with full domain. In \cite{Bausch17,Nguy17} (see 
also 
\cite{Vanden20b}) an extension of  
definition of the classical proximity operator is studied by 
considering
a Bregman distance instead of $\|\cdot\|_U^2$,
under the assumption of uniqueness of the solution to the 
optimization 
problem in \eqref{e:proxgen}. In our context, the 
single-valuedness of 
$\prox^{U}_{f,L}$ is not needed.
The following 
result provides some properties of $\prox_{f,L}^{U}$ in more 
general 
contexts.
\begin{proposition}
\label{p:propgenprox}
Let $f\in\Gamma_0(\HH)$, let $L\colon \HH\to \GG$ be a linear 
bounded operator, and let
$U\colon\GG\to\GG$ be a strongly monotone self-adjoint 
linear
bounded operator. Then, the following hold:
\begin{enumerate}[label=(\roman*)]
\item 
\label{p:propgenproxii}
For every $u\in\dom \prox_{f,L}^{U}$, $L(\prox_{f,L}^{U}u)$, and 
$P_{(\ker 
L)^{\bot}}(\prox_{f,L}^{U}u)$ are singletons.
\item
\label{p:propgenproxiii}
Suppose that $\ker L=\{0\}$. Then, for every $u\in\dom 
\prox_{f,L}^{U}$, $\prox_{f,L}^{U}u$ is a singleton.

\item
\label{p:propgenproxw}
Suppose that $\ran L$ is closed. 
Then
\begin{equation}
(\forall u\in\dom 
\prox_{f,L}^{U})\quad 
\prox_{f,L}^{U}u=\prox_f^{L^*UL} 
(\sqrt{U}L)^{\dagger}\sqrt{U}u.
\end{equation}

\item 
\label{p:propgenproxiv}
Suppose that $\ran L^*=\HH$. Then $\prox_{f,L}^{U}$ is 
single-valued,
$\dom \prox_{f,L}^{U}=\GG$, and
\begin{equation}
(\forall u\in\GG)\quad 
\prox_{f,L}^{U}u=\big\{\prox_f^{L^*UL} 
(L^*UL)^{-1}L^*Uu\big\}.
\end{equation}

\end{enumerate}

\end{proposition}
\begin{proof}	
Let $\mu>0$ be the strong monotonictiy parameter of $U$.
\ref{p:propgenproxii}:
Let $x_1$ and $x_2$ in $\prox_{f,L}^Uu$. It follows from 
\eqref{e:moreoureck} applied to $x_1$ and $x_2$,
the monotonicity of $\partial f$, and strong monotonicity of $U$ 
that
\begin{align}
0&\le 
\scal{-L^*UL(x_1-x_2)}{x_1-x_2}\nonumber\\
&=-\scal{UL(x_1-x_2)}{L(x_1-x_2)}\nonumber\\
&\le -\mu\|L(x_1-x_2)\|^2.
\end{align}
Therefore, $L(x_1-x_2)=0$ which leads to $x_1-x_2\in\ker L$
and, hence, $P_{(\ker L)^{\bot}}x_1=P_{(\ker L)^{\bot}}x_2$.

\ref{p:propgenproxiii}: In this case $(\ker L)^{\bot}=\HH$, which 
yields 
$P_{(\ker L)^{\bot}}=\Id$ and the result follows from 
\ref{p:propgenproxii}.

\ref{p:propgenproxw}: 
{Note that the orthogonal decomposition in 
$(\GG,\scal{\cdot}{\cdot}_{U^{-1}})$ in \eqref{e:orthdecU} and 
\eqref{e:pranU} with $M=L^*$ implies that $U=P_{\ran 
(UL)}^{U^{-1}}U+P_{\ker 
L^*}^{U^{-1}}U$. Therefore, it follows from \eqref{e:moreoureck} 
 that}, for every 
$u\in\GG$ and $x\in\HH$,  
\begin{align}
x\in\prox^{U}_{f,L}u\quad 
&\Leftrightarrow\quad x\in(\partial f+L^*UL)^{-1}L^*P_{\ran 
(UL)}^{U^{-1}}Uu\nonumber\\
&\Leftrightarrow\quad x\in(\partial 
f+L^*UL)^{-1}L^*\big(UL(\sqrt{U}L)^{\dagger}\sqrt{U}^{-1}\big)Uu\nonumber\\
&\Leftrightarrow\quad 
x\in\prox_{f}^{L^*UL}\big((\sqrt{U}L)^{\dagger}\sqrt{U}u\big),
\end{align}
where the last equivalence follows from \eqref{e:proxwarp}.

\ref{p:propgenproxiv}: Note that $\ran L^*=\HH$ yields, for every 
$x\in\GG$, 
$\scal{L^*U Lx}{x}\ge \mu\|Lx\|^2\ge\mu\alpha^2\|x\|^2$,
where the existence of $\alpha>0$ is guaranteed by 
\cite[Fact~2.26]{1}. Therefore, $L^*UL$ is strongly monotone and, 
hence, invertible. Hence, the result follows from 
\ref{p:propgenproxw} 
and $(\sqrt{U}L)^{\dagger}=(L^*UL)^{-1}L^*\sqrt{U}$ in view of 
\eqref{e:moinv}.
\end{proof}
The following result provides sufficient conditions ensuring full 
domain
of $\prox_{f,L}^{U}$. This is a 
 a consequence of Theorem~\ref{p:rescalc} 
in the optimization context and we connect the existence result 
with 
the computation of the proximity operators of $f^*\circ L^*$
and $L\rhd f$.
Our result generalizes 
\cite[Proposition~5.2(iii)]{patrinos2020} to non-standard metrics 
and 
infinite dimensions.

\begin{proposition}
\label{p:proxcalc}
Let $\HH$ and $\GG$ be real Hilbert spaces, let
$f\in\Gamma_0(\HH)$, let 
$L\colon\HH\to\GG$ be a linear bounded operator such that 
\begin{equation}
\label{e:condproxgen}
0\in \sri(\dom f^*-\ran L^*),
\end{equation}
 and let $U\colon\GG\to\GG$ be a 
 strongly monotone self-adjoint linear operator.
Then, the following hold:
\begin{enumerate}[label=(\roman*)]
\item 
\label{p:proxcalc0}
$\dom\prox_{f,L}^{U}=\GG$.
\item 
\label{p:proxcalc1}
$\prox^{U^{-1}}_{f^*\circ L^*}=\Id-U 
 L\,\prox_{f,L}^{U}U^{-1}$.
\item  
\label{p:proxcalc2}
$L\rhd f=(f^*\circ L^*)^*\in\Gamma_0(\HH)$ and $\prox_{L\rhd 
f}^{U}=L\,\prox_{f,L}^{U}$.
\end{enumerate}
\end{proposition}
\begin{proof}
\ref{p:proxcalc0}: 
Since $0\in 
\sri(\dom f^*-\ran L^*)$, \cite[Corollary~16.53(i)]{1} yields $\partial 
(f^*\circ L^*)=L(\partial f^*) L^*$, which is maximally 
monotone in $\HH$ because $f^*\circ L^*\in\Gamma_0(\HH)$ 
\cite[Theorem~20.25]{1}. Hence, by applying
Theorem~\ref{p:rescalc}\ref{p:rescalc1} to $B=\partial 
f^*$ and $M=L^*$, it follows from \eqref{e:moreoureck} that
\begin{align}
\label{e:SyOmega}
(\forall x\in\HH)\quad 
\varnothing\neq((\partial f^*)^{-1}+L^* U L)^{-1} 
L^*Ux
&=(\partial f+L^* U L)^{-1}L^*Ux\nonumber\\
&=\prox^U_{f,L}x.
\end{align}

\ref{p:proxcalc1}: We deduce from \eqref{e:proxU},
Theorem~\ref{p:rescalc}\ref{p:rescalc1}, and 
\eqref{e:SyOmega} 
that
$\prox_{f^*\circ 
L^*}^{U^{-1}}=J_{U\partial (f^*\circ L^*)}=J_{UL(\partial f^*)L^*}
=\Id-U L((\partial 
f^*)^{-1}+L^* U L)^{-1}
L^*=\Id-U L\,\prox^U_{f,L}U^{-1}$.

\ref{p:proxcalc2}: Since $f^*\circ L^*\in\Gamma_0(\HH)$,
\eqref{e:condproxgen} and \cite[Corollary~15.28]{1}  yield
 $L\rhd f=(f^*\circ L^*)^*\in\Gamma_0(\HH)$. Hence, it follows 
from 
\eqref{e:Moreau_nonsme} and \ref{p:proxcalc1} that
\begin{align}
\prox^{U}_{L\rhd f}
&=\Id-U^{-1}\,\prox^{U^{-1}}_{f^*\circ 
L^*}\, U \nonumber\\
&=\Id-U^{-1}\,(\Id-U 
 L\,\prox_{f,L}^{U}U^{-1})\,U\\
&=L\,\prox_{f,L}^{U}
\end{align}
and the proof is complete.
\end{proof}

{Without the qualification condition \eqref{e:condproxgen}, 
$\prox_{f,L}^Uu$ may be empty for some $u\in\GG$, as the 
following 
examples illustrate.
\begin{example}
Suppose that $U=\Id$, that $\ran 
L$ is not closed,
set $f=0$, and let $u\in\overline{\ran L}\setminus \ran L$. Then,
$\inf_{x\in\HH}\|Lx-u\|=0$ but the infimum is not attained.
Observe that, since $f^*=\iota_{\{0\}}$,  we have $\dom 
f^*=\{0\}$ which yields  $\cone(\dom f^*-\ran L^*)=\cone (\ran 
L^*)=\ran L^*\neq  \overline{\ran L^*}=\spc\ran L^*$ and, thus, 
$0\notin\sri (\dom f^*-\ran 
L^*)$.
\end{example}
%The following example shows that existence may fail even in 
%finite 
%dimensions.
\begin{example}
Suppose that $\HH=\RR^2$, $\GG=\RR$, $f\colon (x,y)\mapsto 
\exp(y)$, and $L\colon (x,y)\mapsto x$. Then $L^*\colon 
z\mapsto 
(z,0)$, $\ran L^*=\RR\times\{0\}$, and $f^*\colon (u,v)\mapsto 
\iota_{\{0\}}(u)+\exp^*(v)$, where
\begin{equation}
\exp^*\colon v\mapsto 
\begin{cases}
v(\ln v-1),&\text{if}\:\:v>0;\\
0,&\text{if}\:\:v=0;\\
\pinf,&\text{if}\:\:v<0.
\end{cases}
\end{equation}
Then, $\dom f^*=\{0\}\times \RP$ and $\cone(\dom f^*-\ran 
L^*)=\RR\times\RP\ne\RR^2=\spc(\dom f^*-\ran 
L^*)$, which yields $0\notin\sri (\dom f^*-\ran 
L^*)$.
\end{example}

}
\begin{remark}
\begin{enumerate}[label=(\roman*)]
\item In the case when $U=\mu\Id$, the existence of solutions to 
\eqref{e:proxgen} is assumed in 
\cite[Proposition~5.2(iii)]{patrinos2020} and its 
uniqueness is supposed in 
\cite[Theorem~4.7]{condat2020proximal}.
On the other hand,
the strong monotonicity 
of $(L^*L+\partial f)$ is assumed in \cite{BrediesSun17,Cote17} in 
order 
to guarantee the existence and uniqueness of solutions to the 
optimization problem in \eqref{e:proxgen}. {Under previous 
assumptions, the sequences of ADMM are proved to be well 
defined.}
Proposition~\ref{p:proxcalc}\ref{p:proxcalc0} {provides the 
existence of solutions to \eqref{e:proxgen} under the weaker 
condition 
\eqref{e:condproxgen}.
It is deduced }
from Theorem~\ref{p:proxcalc} and the maximal monotonicity of 
$L(\partial f^*)L^*$, which is obtained from the qualification 
condition 
$0\in \sri(\dom f^*-\ran L^*)$ and $f^*\circ 
L^*\in\Gamma_0(\HH)$ 
in view of \cite[Corollary~16.53(i) \& Theorem~20.25]{1}. 
{Moreover, it follows from 
Proposition~\ref{p:proxcalc}\ref{p:proxcalc2} that
$L\prox^U_{f,L}$ is single-valued, which implies that the 
sequences 
generated by ADMM are well defined. 
In summary,
under the weaker condition $0\in \sri(\dom f^*-\ran L^*)$ we 
guarantee 
the existence and uniqueness of the sequences generated by 
ADMM  
under our approach, generalizing 
\cite{BrediesSun17,condat2020proximal,Cote17,patrinos2020}. A 
general convergence result of ADMM in this context is provided in 
\cite[Theorem~4.6]{Siopt3}.
}

\item In \cite{fadili09,Micc11} fixed point approaches are used in 
order 
to 
compute $\prox_{f\circ L}$ in the context of the sparse recovery 
in 
image processing. This approach leads to sub-iterations in 
optimization algorithms needing to compute $\prox_{f\circ L}$.
{From Proposition~\ref{p:proxcalc}\ref{p:proxcalc1} in the case 
when $U=\Id$, 
our computation is direct once $\prox_{f^*, L^*}$ is easily 
computable.
This is the case, for instance, when $\HH=\RR^n$, $L$ is a 
$n\times 
m$ real matrix with $m>n$, and $f\colon 
x\mapsto x^{\top}Ax/2-z^{\top}x$,
where $A$ is an symmetric positive definite $n\times n$ real 
matrix and $z\in\RR^n$. 
This setting appears, e.g., in signal and image processing 
\cite{CombW05,ista,figueiredo} and 
statistics \cite{lasso1,FusedLASSO,GenLASSO}. In this case, 
the computation of $\prox_{f^*, L^{\top}}$ needs the inversion of 
the 
$n\times n$ real matrix $A^{-1}+LL^{\top}$, while $\prox_{f\circ 
L}$ 
need 
the 
inversion of the $m\times m$ real matrix
$\Id+L^{\top}AL$, which can be more expensive numerically 
when 
$n<<m$. }

\item 
Note that Proposition~\ref{p:proxcalc}\ref{p:proxcalc0} yields
\begin{equation}
\label{e:genMoreau}
\prox^{U^{-1}}_{f^*\circ L^*}+U 
 L\,\prox_{f,L}^{U}U^{-1}=\Id.
\end{equation}
In the case when $L=\Id$, since $\prox_{f,\Id}^{U}=\prox_{f}^{U}$,
\eqref{e:genMoreau} reduces to
\cite[Proposition~24.24(ii)]{1}, which is a non-standard metric 
version 
of Moreau's 
decomposition \cite{Moreau_prox} first derived for mutually polar 
cones \cite{Moreau_cone}.

\end{enumerate}
\end{remark}

\section{Acknowledgements}

The first author thanks the support of ANID under grants 
FON\-DECYT 1190871, Centro de
Modelamiento Matem\'atico (CMM) FB210005
BASAL funds for centers of excellence,  
and grant Redes 180032. The second author thanks the support 
of ANID-Subdirección de Capital  Humano/Doctorado 
Nacional/2018-21181024 and of the Direcci\'on
de Postgrado y Programas from UTFSM through Programa de 
Incentivos a la Iniciaci\'on
Cient\'ifica (PIIC).

% Authors must disclose all relationships or interests that 
% could have direct or potential influence or impart bias on 
% the work: 
%
 \section*{Conflict of interest}
The authors have no conflicts of interest to declare that are 
relevant to 
the content of this article.
%
% The authors declare that they have no conflict of interest.

% BibTeX users please use one of
%\bibliographystyle{spbasic}      % basic style, author-year 
%%citations
%\bibliographystyle{spmpsci}      % mathematics and physical 
%%sciences
%\bibliographystyle{spphys}       % APS-like style for physics
%\bibliography{references}   % name your BibTeX data base

% Non-BibTeX users please use
%\begin{thebibliography}{references}
%
% and use \bibitem to create references. Consult the Instructions
% for authors for reference list style.
%
%\bibitem{RefJ}
% Format for Journal Reference
%Author, Article title, Journal, Volume, page numbers (year)
% Format for books
%\bibitem{RefB}
%Author, Book title, page numbers. Publisher, place (year)
% etc
%\end{thebibliography}

%\bibliography{references2}

\end{document}